\documentclass{article}

\usepackage{wrapfig}
\usepackage[all]{xy}
\usepackage{amsmath}
\usepackage{amsthm}
\usepackage{amssymb}
\usepackage{mathrsfs}
\usepackage{amsthm}

\newtheorem{thm}{Theorem}[section]

\newtheorem{cor}[thm]{Corollary}
\newtheorem{lem}[thm]{Lemma}

\theoremstyle{definition}
\newtheorem{defn}[thm]{Definition}
\newtheorem{rmk}[thm]{Remark}

\newcommand{\sG}{\mathscr{G}}
\newcommand{\sH}{\mathscr{H}}

\newcommand{\cE}{\mathcal{E}}
\newcommand{\cG}{\mathcal{G}}
\newcommand{\cH}{\mathcal{H}}
\newcommand{\cL}{\mathcal{L}}

\newcommand{\cO}{\mathcal{O}}
\newcommand{\cR}{\mathcal{R}}

\newcommand{\C}{\mathbb{C}}

\newcommand{\Q}{\mathbb{Q}}

\newcommand{\Z}{\mathbb{Z}}

\newcommand{\bfB}{\mathbf{B}}
\newcommand{\bfD}{\mathbf{D}}

\newcommand{\Fb}{{\overline{F}}}

\newcommand{\GL}{\mathrm{GL}}

\DeclareMathOperator{\Ker}{\mathrm{Ker}}
\DeclareMathOperator{\Gal}{\mathrm{Gal}}

\DeclareMathOperator{\Hom}{\mathrm{Hom}}
\DeclareMathOperator{\End}{\mathrm{End}}
\DeclareMathOperator{\Ext}{\mathrm{Ext}}

\DeclareMathOperator{\Frac}{\mathrm{Frac}}
\DeclareMathOperator{\Lie}{\mathrm{Lie}}

\newcommand{\ur}{\mathrm{ur}}

\newcommand{\an}{\mathrm{an}}
\newcommand{\pGmod}{(\varphi_q,\Gamma)\mathrm{\mathchar`-mod}/\mathcal{R}}
\newcommand{\pGMod}{\mathrm{Mod}^{\varphi_q,\Gamma}}
\newcommand{\pGModEt}{\mathrm{Mod}^{\varphi_q,\Gamma,\mathrm{\acute{e}t}}}
\newcommand{\Rmod}{\mathcal{R}\mathrm{\mathchar`-mod}}
\newcommand{\id}{\mathrm{id}}
\newcommand{\LT}{\mathrm{LT}}

\newcommand{\oc}{\mathrm{oc}}

\DeclareMathOperator{\Rep}{\mathrm{Rep}}

\newcommand{\isomto}{\xrightarrow{\sim}}

\newcommand{\cln}{\colon}

\newcommand{\smin}{\smallsetminus}

\newcommand{\vphi}{\varphi}
\newcommand{\vpi}{\varpi}

\def\maru#1{{\ooalign{\hfil
			\ifnum#1>999 \resizebox{.25\width}{\height}{#1}\else%
			\ifnum#1>99 \resizebox{.33\width}{\height}{#1}\else%
			\ifnum#1>9 \resizebox{.5\width}{\height}{#1}\else #1%
			\fi\fi\fi%
			\/\hfil\crcr%
			\raise.167ex\hbox{\mathhexbox20D}}}}

\usepackage{mathtools}
\mathtoolsset{showonlyrefs = true}

\author{Megumi Takata\thanks{Kyushu Sangyo University,
		3-1 Matsukadai 2-chome,
		Higashi-ku, Fukuoka
		813-8503 Japan. \newline E-mail address: \texttt{m.takata@ip.kyusan-u.ac.jp}
		}}
\title{On a relation of overconvergence and $F$-analyticity on $p$-adic Galois representations of a $p$-adic field $F$}

\begin{document}
	
	
	\maketitle
	
	\begin{abstract}
		Let $p$ be a prime number.
		There are properties called ``overconvergence'' and ``$F$-analyticity'' for $p$-adic Galois representations of a $p$-adic field $F$.
		By Berger's work, it is known that $F$-analyticity is stricter than  overconvergence.
		In this article, we show that, in many cases, an overconvergent Galois representation is $F$-analytic up to a twist by a character.
		This result emphasizes the necessity of the theory of $(\varphi,\Gamma)$-modules over the multivariable Robba ring, by which we expect to study all $p$-adic Galois representations.
	\end{abstract}

	\section{Introduction}
	
	In the $p$-adic local Langlands program for $\GL_2(\Q_p)$, which has finally been established in \cite{Colmez2010b} and \cite{CDP2014},
	an important result in the early stage is a theorem due to Cherbonnier and Colmez \cite{CC1998}.
	It says that all of the $p$-adic Galois representations of a $p$-adic field are \textit{overconvergent} with respect to the cyclotomic $\Z_p$-extension.
	It enables us to study every $p$-adic Galois representation of a $p$-adic field in the terms of modules over the univariable Robba ring, which is the ring of (not necessarily bounded) functions on $p$-adic annuli.
	One of the benefits appears in Berger's work \cite{Berger2002}, which relates the following three notions:
	1) $(\vphi,\Gamma)$-modules,
	2) $p$-adic differential equations, and
	3) filtered $(\vphi, N)$-modules.
	
	Thus, if we want to generalize the $p$-adic local Langlands program from $\GL_2(\Q_p)$ to $\GL_2(F)$ for an arbitrary $p$-adic field $F$,
	a similar question occurs:
	is any $p$-adic Galois representation of $F$ overconvergent if we consider a general Lubin-Tate extension instead of the cyclotomic one?
	For this, Fourquaux and Xie \cite{FX2013} show a negative answer.
	That is, contrary to the cyclotomic case, there exist (an infinite number of) $p$-adic Galois representations which is not overconvergent with respect to a Lubin-Tate extension.
	Then we face a next question: how strict is  overconvergence in a general Lubin-Tate setting?
	Berger \cite{Berger2016} has proved that, an another property called \textit{$F$-analyticity} is a sufficient condition for  overconvergence.
	In addition, he has proved the following:
	\begin{thm}[{\cite[Cor. 4.3]{Berger2013}}]\label{PreviousResearchOfBerger}
		Any absolutely irreducible and overconvergent $p$-adic Galois representations of $F$ are $F$-analytic up to twist by a character.
	\end{thm}
	We remark that Berger and Fourquaux \cite[Th.\ 1.3.1, Cor.\ 1.3.2]{BF2017} describe an arbitrary overconvergent representation by using an $F$-analytic one and one which factors through $\Gamma$, where $\Gamma$ is the Galois group of the considering Lubin-Tate extension.
	
	The main theorem of this article is a generalizaion of Theorem \ref{PreviousResearchOfBerger}.
	To state it in detail, we need to introduce some notations.
	We fix an algebraic closure $\Fb$ of $F$ and write $\sG_F = \Gal(\Fb/F)$.
	Let $L \subset \Fb$ denote a finite extension of $F$.
	Let $F_\vpi^\LT$ be the Lubin-Tate extension of $F$ with respect to a fixed uniformizer $\vpi \in F$.
	We write $\Gamma = \Gal(F_\vpi^\LT/F)$.
	Let $q$ be the order of the residue field of $F$.
	Let $\cR$ be the Robba ring with coefficients in $L$.
	It has an action of $\Gamma$ and $\vphi_q$, where $\vphi_q$ is a lift of the $q$-th power map.
	We have a fully faithful functor $\cR\otimes_{\cE^\dagger}\bfD^\dagger$ from the category of $L$-representations of $\sG_F$
	to that of the $(\vphi_q, \Gamma)$-modules over $\cR$
	 (for more precise descriptions of $\cR$, the notion of $(\vphi_q, \Gamma)$-modules, and $\cR\otimes_{\cE^\dagger}\bfD^\dagger$, see \S 2).
	\begin{thm}\label{MainThm}
		Let $V$ be an overconvergent $L$-representation of $\sG_F$.
		Suppose that $D = \cR\otimes_{\cE^\dagger}\bfD^\dagger(V)$ has a filtration
		\[
		0 = D_0 \subset D_1 \subset \cdots \subset D_r = D
		\]
		of $(\vphi_q, \Gamma)$-modules over $\cR$.
		We put $\Delta_i = D_i/D_{i-1}$.
		We assume that the following conditions hold.
		
		\begin{enumerate}
			\renewcommand{\theenumi}{{\rm (\alph{enumi})}}
			\item For any $1 \leq i \leq r,$ we have $\End_{\pGmod} (\Delta_i) = L$.
			\item For any $1 \leq i<j \leq r$, we have $\Hom_{\pGmod} (\Delta_j,\Delta_i) = 0$.
			\item For any $1 < i \leq r,$ the short exact sequence 
			$ 0 \to D_{i-1} \to D_i \to \Delta_i \to 0$ does not split.
		\end{enumerate}
		Then there exist a finite extension $L'$ of $L$ and a character $\delta\cln \sG_F^\times \to (L')^\times$ such that
		$V\otimes_L L'(\delta)$ is $F$-analytic.
	\end{thm}

	\begin{rmk}
		\begin{enumerate}
			\renewcommand{\theenumi}{{(\roman{enumi})}}
			\item The case $r=1$ of Theorem \ref{MainThm} says that, if an overconvergent $L$-representation $V$ of $\sG_F$ satisfies $\End(V) =L$, then $V$ becomes $F$-analytic after twisting by a character.
			It is still a generalization of Theorem \ref{PreviousResearchOfBerger}.
			\item Off course, there are overconvergent representations which can not be $F$-analytic after twisting by any character:
			for example, the direct sum of $F$-analytic one and not $F$-analytic one.
			However, a large part of overconvergent representations seem to satisfy the conditions of Theorem \ref{MainThm}:
			for example, most of trianguline representations.
			While, $F$-analyticity is very strict since, by definition, it demands the Hodge-Tate weights with respect to any $\tau\in \Hom(F,\Fb)\smin \{\id_F\}$
			to be zero.
			Hence this result predicts that the size of the class of overconvergent representations is thought to be very small in the whole.
			Therefore we seem to need the theory of \textit{$(\varphi,\Gamma)$-modules over the multivariable Robba ring} (\cite{Berger2013}, \cite{Berger2016}, \cite{BF2017}), by which we expect to study all $p$-adic Galois representations.
		\end{enumerate}
	\end{rmk}
	
	In \S 2, we recall some definitions and theorems on the Lubin-Tate $(\vphi_q,\Gamma)$-modules.
	In particular, we define overconvergence and $F$-analyticity here.
	In \S 3, after introducing several lemmas, we prove Theorem \ref{trueMainThm}.
	Then we obtain Theorem \ref{MainThm} as a corollary.
	
	\subsubsection*{Acknowledgments}
	The author would like to thank Kentaro Nakamura for many useful comments, in particular, for suggesting him to generalize the original version of the main theorem.
	
	\section{Preliminary on the Lubin-Tate $(\vphi_q,\Gamma)$-modules}
	
	In this section, we recall some notion on the Lubin-Tate $(\vphi_q,\Gamma)$-modules.
	
	Let $F$ be a finite extension of $\Q_p$, $\cO_F$ the ring of integers of $F$, and $q$ the order of the residue field of $\cO_F$.
	We fix an algebraic closure $\Fb$ of $F$.
	We write $\sG_F = \Gal(\Fb/F)$.
	
	We fix a uniformizer $\vpi$ of $F$ and a power series $f_\vpi(T) \in \cO_F[\![T]\!]$ such that
	$f_\vpi(T) \equiv \vpi T$ modulo degree 2 and $f_\vpi(T) \equiv T^q$ modulo $\vpi$.
	Then there exists a unique formal group law $\cG_\vpi(X,Y) \in \cO_F[\![ X,Y ]\!]$ such that $f_\vpi \in \End(\cG_\vpi)$,
	which we call \textit{the Lubin-Tate formal group associated to $\vpi$}.
	It has a natural formal $\cO_F$-module structure $[\cdot]\cln \cO_F \to \End(\cG_\vpi)$ such that, for any $a\in \cO_F$, the first term of $[a](T)$ is $aT$ and $[\vpi](T) = f_\vpi(T)$.
	For any $a\in \cO_F$, we put $\cG_\vpi[a] = \{ \alpha\in \Fb \mid [a](\alpha) = 0 \}.$
	We denote by $F_\vpi^\LT$ the extension of $F$ in $\Fb$ obtained by adding all elements of $\cG_\vpi[\vpi^n]$ for all integers $n\geq 1$.
	We call $F_\vpi^\LT$ \textit{the Lubin-Tate extension of $F$ associated to $\vpi$}.
	We define the Tate module $T\cG_\vpi$ of $\cG_\vpi$ by $T\cG_\vpi = \varprojlim_n \cG[\vpi^n].$
	This is a free $\cO_F$-module of rank 1 on which $\sG_F$ naturally acts.
	It induces a character $\chi_\LT\cln \sG_F \to \cO_F^\times$, which we call \textit{the Lubin-Tate character associated to $\vpi$}.
	We write $\sH_F = \Gal(\Fb/F_\vpi^\LT)$, which is the kernel of $\chi_\LT$.
	For any integer $n\geq 1$, we put
	\begin{align*}
	\Gamma &= \Gal(F_\vpi^\LT/F) = \sG_F/\sH_F \isomto \cO_F^\times \quad \mbox{and} \\
		\Gamma_n &= \Gal(F_\vpi^\LT/F(\cG_\vpi[\vpi^n])).
	\end{align*}
	
	Let $L \subset \Fb$ be a finite extension of $F$.
	We put
	\[
		\cO_\cE = \cO_{\cE_L} = \left\{ \sum_{n\in \Z} a_n T^n \ \middle| \  a_n\in \cO_L \mbox{ for any $n\in \Z$}, \ a_n \to 0 \ (n\to -\infty) \right\},
	\]
	which is a discrete valuation ring such that a uniformizer $\vpi_L$ of $L$ generates the maximal ideal.
	We put $\cE = \cE_L = \Frac(\cO_\cE)$.
	In this article, we consider \textit{the weak topology}.
	For any $n\in \Z$, we define a topology on $\vpi_L^n\cO_\cE$ with $\{\vpi_L^{i+n}\cO_\cE + T^j\vpi_L^n\cO_L [\![T]\!] \}_{i,j}$ as a fundamental system of neighborhoods of 0.
	We define a topology on $\cE = \cup_{n\in\Z} \vpi_L^n \cO_\cE$ by the inductive limit topology.
	
	An $L$-linear $(\vphi_q,\Gamma)$-action on $\cO_\cE$ or on $\cE$ is defined such that, for any $\gamma\in \Gamma$, we have
	\[
		\vphi_q(T) = f_\vpi(T), \quad \gamma.T = [\chi_\LT(\gamma)](T).
	\]
	A free $\cO_\cE$-module $D_0$ of finite rank is called a \textit{$(\vphi_q,\Gamma)$-module over $\cO_\cE$} if:
	\begin{enumerate}
		\item $\Gamma$ acts continuously and semilinearly on $D_0$ and
		\item a $\vphi_q$-semilinear map $\varPhi\cln D_0 \to D_0$ is equipped such that
		\begin{itemize}
			\item the action of $\Gamma$ and $\varPhi$ commute, and
			\item the $\cO_\cE$-linear homomorphism $\cO_\cE \otimes_{\vphi_q,\cO_\cE} D_0 \to D_0$ induced by $\varPhi$ is an isomorphism.
		\end{itemize}
	\end{enumerate}
	We abuse notation and use the same symbol $\vphi_q$ to denote $\varPhi$.
	The notion of \textit{$(\vphi_q,\Gamma)$-modules over $\cE$} is defined in a similar way as above.
	We call a $(\vphi_q,\Gamma)$-module $D$ over $\cE$ \textit{\'etale} if there exists a $(\vphi_q,\Gamma)$-module $D_0$ over $\cO_\cE$ such that $D \simeq \cE\otimes_{\cO_\cE} D_0$ as $(\vphi_q,\Gamma)$-modules over $\cE$.
	
	We denote by $\Rep_L (\sG_F)$ the category of continuous finite dimensional $L$-representations of $\sG_F$
	and by $\pGMod_{/\cE}$ the category of $(\vphi_q,\Gamma)$-modules over $\cE$.
	Let $\pGModEt_{/\cE}$ denote the full subcategory of $\pGMod_{/\cE}$ consisting of \'etale objects.
	Fontaine have found an equivalence of categories of $\Rep_L (\sG_F)$ and $\pGModEt_{\cE}$ in the cyclotomic case.
	In the general Lubin-Tate cases, it is given by Kisin and Ren.
	To construct it, we use a big field $\bfB$, which contains $\cE_F$ and has a ($\vphi_q,\sG_F$)-action.
	Moreover, we have $\bfB^{\sH_F} = \cE_F$.
	For a precise definition of $\bfB$, the reader may refer to \cite{KR2009} or \cite[1B]{FX2013}.
	Note that, in \cite{KR2009} (resp. \cite{FX2013}), this is denoted by $\widehat{\cE}^\ur$ (resp. $\mathrm{B}$).
	\begin{thm}[{\cite[Th.\ 3.4.3, Rem.\ 3.4.4]{Fon1991}, \cite[Th.\ 1.6]{KR2009}}]
		We have a functor
		\[
			\bfD \cln \Rep_L (\sG_F) \to \pGModEt_{/\cE} \cln V\mapsto (V\otimes_{F}\bfB)^{\sH_F},
		\]
		which gives an equivalence of categories.
		Its quasi-inverse functor is given by $D \mapsto (D\otimes_{\cE_F} \bfB)^{\vphi_q=1}$.
	\end{thm}
	
	Now we define a subfield $\cE^\dagger$ of $\cE$ and a ring $\cR$ containing $\cE^\dagger$.
	Let $v_p$ denote the $p$-adic additive valuation on $\Fb$ normalized as $v_p(p)=1$.
	For any formal series $f(T) = \sum_{n\in \Z} a_n T^n$ with coefficients in $L$ and any real number $r\geq 0$, we define $v^{\{r\}}(f) = \inf_n v_p(a_n) + nr$, which may be $\pm \infty$.
	For any $0\leq s \leq r$, we put $v^{[s,r]}(f) = \inf_{s\leq r' \leq r} v^{\{r'\}}(f)$.
	We define
	\begin{align*}
		\cE^{[s,r]} = \cE^{[s,r]} _L &= \left\{ f(T) = \sum_{n\in \Z} a_n T^n \ \middle| \ 
			\begin{array}{l}
				a_n \in L \mbox{ for any $n\in \Z$}, \\
				v^{[s,r]}(f) \neq -\infty
			\end{array} \right\}, \\
		\cE^{]0,r]} = \cE^{]0,r]}_L &= \bigcap_{0 < s \leq r} \cE^{[s,r]}, \\
		\cR = \cR_L &= \bigcup_{r>0} \cE^{]0,r]}, \\
		 \cE^{(0,r]} = \cE^{(0,r]}_L &= \cE^{]0,r]} \cap \cE, \mbox{ and} \\
		 \cE^\dagger = \cE_L^\dagger &= \bigcup_{r>0} \cE^{(0,r]}.
	\end{align*}
	They have ring structure and $\Gamma$ acts on them such that $\gamma.T = [\chi_\LT(\gamma)](T)$ for any $\gamma\in \Gamma$.
	Moreover, if $r$ is sufficiently small, then we have a ring endomorphism
	$\vphi_q\cln \cE^{[s,r]} \to \cE^{[sq,rq]}$ such that $\vphi_q(T) = f_\vpi(T)$.
	Hence the rings $\cE^\dagger$ and $\cR$ are endowed with a $(\vphi_q,\Gamma)$-action.
	We call $\cR$ \textit{the Robba ring} and $\cE^\dagger$ \textit{the bounded Robba ring}.
	We have inclusions $\cE \supset \cE^\dagger \subset \cR$ as rings with a $(\vphi_q,\Gamma)$-action.
	
	The ring $\cE^{[s,r]}$ is equipped with the topology defined by the valuation $v^{[s,r]}$.
	Then $\cE^{[s,r]}$ is complete.
	We endow $\cE^{]0,r]}$ with the projective limit topology, and $\cR$ with the inductive limit topology.
	On $\cE^{(0,r]}$, we consider the topology defined by the valuation $v^{\{r\}}$.
	We endow $\cE^\dagger$ with the inductive limit topology.
	
	The notion of \textit{$(\vphi_q,\Gamma)$-modules over $\cE^\dagger$} or $\cR$ is defined in a similar way as those over $\cO_\cE$ or $\cE$.
	A $(\vphi_q,\Gamma)$-modules $D^\dagger$ over $\cE^\dagger$ is called \textit{\'etale} if $D^\dagger\otimes_{\cE^\dagger}\cE$ is \'etale.
	A $(\vphi_q,\Gamma)$-modules $D$ over $\cR$ is called \textit{\'etale} or \textit{of slope 0} if there exists an \'etale $(\vphi_q,\Gamma)$-module $D^\dagger$ over $\cE^\dagger$ such that $D \simeq \cR \otimes_{\cE^\dagger} D^\dagger$.
	Let $\pGMod_{/\cE^\dagger}$ (resp. $\pGMod_{/\cR}$) denote the category of $(\vphi_q,\Gamma)$-modules over $\cE^\dagger$ (resp. $\cR$).
	We denote by $\pGModEt_{/\cE^\dagger}$ (resp. $\pGModEt_{/\cR}$) the full subcategory of $\pGMod_{/\cE^\dagger}$ (resp. $\pGMod_{/\cR}$) consisting of \'etale objects.
	
	\begin{thm}[{\cite[Prop.\ 1.6]{FX2013}}]
		The functor $D^\dagger \mapsto \cR\otimes_{\cE^\dagger} D^\dagger$ gives an equivalence of categories of $\pGModEt_{/\cE^\dagger}$ and $\pGModEt_{/\cR}$.
	\end{thm}
	
	We have a subfield $\bfB^\dagger$ of $\bfB$ as in \cite[1B]{FX2013},
	which contains $\cE_F^\dagger$ and has a ($\vphi_q,\sG_F$)-action.
	Moreover, we have $(\bfB^\dagger)^{\sH_F} = \cE_F^\dagger$.
	For any object $V$ of $\Rep_L (\sG_F)$, we put $\bfD^\dagger(V) = (V\otimes_F \bfB^\dagger)^{\sH_F}$.
	We have a natural inclusion $\bfD^\dagger(V) \subset \bfD(V)$.
	
	\begin{defn}
		An object $V$ of $\Rep_L (\sG_F)$ is \textit{overconvergent} if $\bfD^\dagger(V)$ generates $\bfD(V)$ as an $\cE$-vector space.
		We denote by $\Rep_L^{\oc}(\sG_F)$ the subcategory of $\Rep_L(\sG_F)$ consisting of overconvergent objects.
	\end{defn}
	
	\begin{thm}[{\cite[Prop.\ 1.5]{FX2013}}]
		The operation $\bfD^\dagger$ gives an equivalence of categories of $\Rep_L^\oc(\sG_F)$ and $\pGModEt_{\cE^\dagger}$.
		Moreover, the diagram
		\[
		\xymatrix{
			\Rep_L^{\oc}(\sG_F) \ar[r]^{\bfD^\dagger} \ar@{^(->}[d] & \pGModEt_{/\cE^\dagger} \ar[d]^{\cE \otimes_{\cE^\dagger} (\cdot)} \\
			\Rep_L(\sG_F) \ar^{\bfD}[r] & \pGModEt_{/\cE}
		}\]
		commutes up to canonical isomorphisms.
	\end{thm}
	
	\begin{rmk}
		Cherbonnier and Colmez \cite{CC1998} have shown that, in the cyclotomic case, we have $\Rep_L^\oc (\sG_F) = \Rep_L (\sG_F)$, i.e. all of the $L$-representations of $\sG_F$ are overconvergent with respect to the cyclotomic extension.
		Even in the general Lubin-Tate case, all of the 1-dimensional $L$-representations of $\sG_F$ are overconvergent \cite[Remark\ 1.8]{FX2013}.
		However, there exist $L$-representations of $\sG_F$ which are not overconvergent,
		as shown by Fourquaux and Xie \cite[Theroem\ 0.6]{FX2013}.
	\end{rmk}
	
	Now, let us recall another property of $L$-representations of $\sG_F$ so-called $F$-analyticity.
	We write $\C_p$ for the $p$-adic completion of $\Fb$.
	\begin{defn}
		An object $V$ of $\Rep_L(\sG_F)$ is called $F$-analytic if, for any $\tau\in \Hom_{\Q_p\mathrm{\mathchar`-alg}}(F,\Fb)\smin \{\id_F\}$, the $\C_p$-representation $\C_p\otimes_{\tau, F} V$ is trivial.
	\end{defn}
	
	For $F$-analytic representations, Berger shows the following:
	
	\begin{thm}[{\cite[Thm.\ C]{Berger2016}}]
		Any $F$-analytic $L$-representation of $\sG_F$ is overconvergent.
	\end{thm}
	
	There is also a notion of $F$-analyticity for $(\vphi_q,\Gamma)$-modules over $\cR$.
	To define it, we introduce an action of a Lie algebra.
	Let $\Lie\Gamma$ denote the Lie algebra associated to the $p$-adic Lie group $\Gamma$.
	Note that there is an isomorphism $\Lie\Gamma \isomto \cO_F$ induced by $\chi_\LT$.
	Let $D$ be any object in $\pGMod_{/\cR}$.
	We will define an action of $\Lie\Gamma$ on $D$.
	Take any $x\in D$.
	If $\beta \in \Lie\Gamma$ is sufficiently close to 0, then we can define $\gamma = \exp\beta \in \Gamma$.
	Now we recall the following:
	\begin{lem}[{\cite[Lemma 1.7]{FX2013}}]
		For any $r>0$ and $0<s \leq r$, there exists $n=n(s,t)$ such that,
		for any $\gamma' \in \Gamma_n$ and $f\in \cE^{[s,r]}$, we have $v^{[s,r]}((\gamma'-1)f) \geq v^{[s,r]}(f) + 2$.
	\end{lem}
	Note that, in the original version, the right hand side is $v^{[s,r]}(f) + 1$.
	However, by the proof, we can replace 1 with any other positive real number, for example 2.

	We fix a basis $e_1,\ldots,e_d$ of $D$.
	For any $r> 0$, we put
	$
	D^{]0,r]} = \oplus_{i=1}^d \cE^{]0,r]} e_i.
	$
	It depends on the choice of $e_1,\ldots,e_d$.
	However, for another choice $e_1',\ldots,e_d'$, there exists $0< r' < r$ such that $\oplus_{i=1}^d \cE^{]0,r']} e_i = \oplus_{i=1}^d \cE^{]0,r']} e_i'.$
	We put $D^{[s,r]} = \cE^{[s,r]}\otimes_{\cE^{]0,r]}}D^{]0,r]}.$
	\begin{cor}\label{valuationLemmaGeneralizedVersion}
		For any $r>0$ and $0<s \leq r$, there exists $n=n(s,t)$ such that,
		for any $\gamma' \in \Gamma_n$ and $x\in D^{[s,r]}$, we have $v^{[s,r]}((\gamma' -1)x) \geq v^{[s,r]}(x) + 2$.
	\end{cor}
	\begin{proof}
		Since we can prove it in the same way as \cite[Lem.\ 5.2]{Berger2002}, we omit it.
	\end{proof}
	We choose an integer $m\geq 0$ such that $\gamma^{p^m} \in \Gamma_{n(s,t)}$.
	By Corollary \ref{valuationLemmaGeneralizedVersion}, the series
	\[
		(\log \gamma)_{[s,r]} (x)  = \frac{1}{p^m}\sum_{i=1}^\infty (-1)^{i-1}\frac{(\gamma^{p^m} - 1)^i}{i}x
	\]
	converges in $D^{[s,r]}$, and it is independent of the choice of $m$.
	We can show that, if $s' < s$, then the image of $(\log\gamma)_{[s',r]}(x)$ via the natural map $D^{[s',r]} \to D^{[s,r]}$ coincides with $(\log\gamma)_{[s,r]}(x)$.
	Thus we obtain a projective system $((\log \gamma)_{[s,r]})_{0<s\leq r}$,
	which gives an element $(\log\gamma) x \in D^{]0,r]}$.
	We put $d\Gamma_\beta x = (\log \gamma)x$.
	For a general $\beta \in \Lie\Gamma$, we choose a sufficiently large $n \in \Z$ and put $d\Gamma_\beta x = p^{-n}d\Gamma_{p^n\beta}x$,
	which is independent of the choice of $n$.
	As a result, we have an action $d\Gamma_\bullet \cln \Lie\Gamma \to \End_L(D)$.

	For an element $\beta \in \Lie\Gamma$ sufficiently close to 1, $\nabla_\beta|_{D}$ denotes the operator $(\log\chi_\LT(\exp\beta))^{-1}d\Gamma_\beta$ on $D$.
	
	\begin{defn}
		An object $D$ of $\pGMod_{/\cR}$ is called \textit{$F$-analytic} if the operator $\nabla_\beta|_D$
		is independent of the choice of $\beta \in \Lie\Gamma$.
		If so, we often denote $\nabla_\beta|_D$ by $\nabla|_D$, or $\nabla$ if no confusion occurs.
	\end{defn}
	
	We have a formula to calculate $\log\gamma$ as follows:
	
	\begin{lem}\label{calculationFormulaOfLog}
		For any $x\in D$ and an element $\gamma\in \Gamma$ sufficiently close to 1, we have
		\[
		(\log \gamma) x = \lim_{n\to \infty} \frac{\gamma^{p^n}x - x}{p^n}.
		\]
	\end{lem}
	
	\begin{proof}
		It suffices to show the formula for $D^{[s,r]}$.
		We fix a basis $e_1,\ldots,e_d$ of $D$.
		For any $y=\sum_{i=1}^{d} y_i e_i \in D^{[s,r]} = \oplus_{i=1}^d \cE^{[s,t]}e_i$, we define $v^{[s,r]}(y) = \inf_{1\leq i \leq d} v^{[s,r]}(y_i)$.
		Let $\cL$ be an $L$-linear operator on $D^{[s,r]}$.
		We define a valuation of $\cL$ by
		\[
			v^{[s,r]}(\cL) = \inf_{y\in D^{[s,r]}} (v^{[s,r]}(\cL y) - v^{[s,r]}(y)).
		\]
		If $v^{[s,r]}(\cL) > (p-1)^{-1}$, then the operator
		\[
		\exp \cL = \sum_{i=0}^{\infty} \frac{\cL^i}{i!}
		\]
		on $D^{[s,r]}$ is well-defined.
		
		By corollary \ref{valuationLemmaGeneralizedVersion}, if an element $\gamma\in \Gamma$ is sufficiently close to 1,
		then the operator $\log \gamma$ on $D^{[s,r]}$ is well-defined and 
		$v^{[s,r]}(\log \gamma) \geq 2$.
		Thus we can define $\exp(\log\gamma)$.
		Moreover, we have
		\[
		\gamma x = [\exp(\log \gamma)](x).
		\]
		Hence, for any integer $n\geq 1$, we have
		\begin{align}
		\gamma^{p^n} x &= [\exp(p^n \log \gamma)](x) \\
		&= \sum_{i=0}^{\infty} \frac{(p^n \log \gamma)^i}{i!} x, \quad \mbox{and} \\ 
		\frac{\gamma^{p^n}x - x}{p^n} &= (\log \gamma)(x) + p^n \sum_{i=0}^{\infty} \frac{p^{ni}(\log\gamma)^{i+2}}{(i+2)!} x.  \label{gamma_pow_p_to_the_n} 
		\end{align}
		Note that the second term of the right-hand side of \eqref{gamma_pow_p_to_the_n} is well-defined since $v^{[s,r]} \left(p^{ni}(\log\gamma)^{i+2}x/(i+2)! \right)$ tends to 0 as $i\to \infty$.
		As $n$ goes to $\infty$ in \eqref{gamma_pow_p_to_the_n}, we have the required equality.
	\end{proof}
	
	Berger shows that $F$-analyticity of an $L$-representation of $\sG_F$ is equivalent to that of the corresponding $(\vphi_q,\Gamma)$-module.
	\begin{thm}[{\cite[Thm.\ D]{Berger2016}}]\label{equivBetTwoAnalyticity}
		Let $V$ be any object of $\Rep_L^\oc(\sG_F)$.
		Then $V$ is $F$-analytic if and only if $\cR\otimes_{\cE^\dagger} \bfD^\dagger(V)$ is $F$-analytic.
	\end{thm}
	
	\section{Proof of the main theorem}
	In this section, we introduce several lemmas first.
	Next, we prove Theorem \ref{trueMainThm}, which is a $(\vphi_q,\Gamma)$-module version of the main theorem.
	Then we obtain Theorem \ref{MainThm} as a corollary.
	
	Let $D_1$ and $D_2$ be $(\vphi_q, \Gamma)$-modules over $\cR$.
	We denote by $\Ext(D_1,D_2)$ the set of isomorphism classes of extensions of $D_1$ by $D_2$ in the category of $(\vphi_q, \Gamma)$-modules over $\cR$, which has a natural $L$-vector space structure.
	\begin{lem}\label{nonTrivExt}
		Let $\Delta$ and $D$ be $(\vphi_q, \Gamma)$-modules over $\cR$.
		Suppose that $D$ is $F$-analytic and \[\End_{\pGmod}(\Delta) = L.\]
		If $\Ext(\Delta, D)\neq 0$, then $\Delta$ is $F$-analytic.
	\end{lem}
	\begin{proof}
		This lemma is a generalization of \cite[Theorem 5.20]{FX2013}.
		Suppose that $\Delta$ is not $F$-analytic.
		Then there exist $\beta,\ \beta' \in \Lie\Gamma$ such that
		$\nabla_\beta|_{\Delta} \neq \nabla_{\beta'}|_{\Delta}$.
		We put $\nabla'|_{\Delta} = \nabla_\beta|_{\Delta} - \nabla_{\beta'}|_{\Delta}$.
		Since $\nabla_\beta|_{\Delta}$ and $\nabla_{\beta'}|_{\Delta}$ are both $\cR$-derivations and the trivial $(\vphi_q, \Gamma)$-module $\cR$ is $F$-analytic, the operator $\nabla'|_{\Delta}$ is $\cR$-linear.
		Moreover, it is stable under $(\vphi_q,\Gamma)$-action.
		Hence we have $\nabla'|_{\Delta} \in \End_{\pGmod}(\Delta)$, and by assumption, there exists $c\in L^\times$ such that
		$\nabla'|_{\Delta} = c \cdot \id_\Delta.$ 
		Now we choose any extension $0 \to D \to \widetilde{D} \to \Delta \to 0$,
		and consider the operator $\nabla'|_{\widetilde{D}} = \nabla_\beta|_{\widetilde{D}} - \nabla_{\beta'}|_{\widetilde{D}}$,
		which is also $\cR$-linear and $(\vphi_q,\Gamma)$-stable.
		Since $D$ is $F$-analytic, we have $D \subset \Ker (\nabla'|_{\widetilde{D}})$ and $\nabla'|_{\widetilde{D}}$ induces an $\cR$-linear and $(\vphi_q,\Gamma)$-stable homomorphism $\Delta \to \widetilde{D}$, which is a section of $\widetilde{D} \to \Delta.$
		Therefore we have $\End(\Delta,D) = 0$ as conclusion.
	\end{proof}	
	
	Let $\Delta$ and $D$ be $(\vphi_q, \Gamma)$-modules over $\cR$.
	Then $\Hom_{\Rmod}(\Delta, D)$ has a canonical $(\vphi_q, \Gamma)$-module structure, and we have a natural isomorphism
	\begin{equation}
		\Ext(\cR, \Hom_{\Rmod}(\Delta,D)) \isomto \Ext(\Delta, D). \label{isomBetweenTwoExts}
	\end{equation}
	
	Here, we describe the isomorphism \eqref{isomBetweenTwoExts} explicitly.
	Take any extension \[ 0 \to \Hom_{\Rmod}(\Delta, D) \to \widetilde{\cH} \to \cR \to 0\]
	and choose a lift $\widetilde{h} \in \cH$ of $1\in \cR$.
	Then both $(\vphi_q-1) \widetilde{h}$ and $\ (\gamma - 1) \widetilde{h}$ are in $\Hom_{\Rmod}(\Delta, D)$.
	We put $\widetilde{D} = D \oplus \Delta $ as an $\cR$-module,
	and define a  $(\vphi_q,\Gamma)$-action on $\widetilde{D}$ as follows:
	on $D \subset \widetilde{D}$, we use the original $(\vphi_q,\Gamma)$-action.
	For any $x\in \Delta$ and $\gamma \in \Gamma$, we define
	\begin{align*}
		\vphi_q |_{\widetilde{D}} (x) &= \vphi_q|_{\Delta} (x) + [(\vphi_q-1)\widetilde{h}](x), \\
		\gamma|_{\widetilde{D}} (x) &= \gamma|_{\Delta}(x) +  [(\gamma-1)\widetilde{h}](x).
	\end{align*}
	Here, for two $\cR$-modules $V_1 \subset V_2$, an element $x\in V_1$ and an operator $T$ which acts on both $V_1$ and $V_2$, the notation $T|_{V_i}(x)$ means the image of $x$ by $T$ regarding $x$ as an element in $V_i$.
	Since the isomorphism class $[\widetilde{D}]$ of the $(\vphi_q,\Gamma)$-module $\widetilde{D}$ is independent of the choice of $\widetilde{\cH}$ and $\widetilde{h}$, we have a map
	\[ \Ext(\cR, \Hom_{\Rmod}(\Delta,D)) \to \Ext(\Delta, D) \cln [\widetilde{\cH}] \mapsto [\widetilde{D}], \]
	which is the isomorphism (\ref{isomBetweenTwoExts}).
	
	Now we present the inverse map of \eqref{isomBetweenTwoExts}.
	Take any extension $0 \to D \to \widetilde{D} \to \Delta \to 0$ and choose a section $s\cln \Delta \to \widetilde{D}$ as an $\cR$-module.
	Note that $s$ is an element of the $(\vphi_q,\Gamma)$-module $\Hom_{\Rmod}(\Delta,\widetilde{D})$.
	Then we have 
	\[
		(\vphi_q-1)s,\ (\gamma-1)s \in \Hom_{\Rmod}(\Delta,D)
	\]
	for any $\gamma \in \Gamma$.
	We put $\widetilde{\cH} = \Hom_{\Rmod}(\Delta,D) \oplus \cR \widetilde{h}$ as an $\cR$-module, where $\cR \widetilde{h}$ is a free $\cR$-module of rank 1 of which $\widetilde{h}$ gives a basis.
	We define a $(\vphi_q,\Gamma)$-action on $\widetilde{\cH}$ as follows:
	on $\Hom_{\Rmod}(\Delta,D) \subset \widetilde{\cH}$, we use the original $(\vphi_q,\Gamma)$-action.
	We define
	\begin{align*}
		\vphi_q \widetilde{h} & = \widetilde{h} + (\vphi_q - 1) s, \\
		\gamma \widetilde{h} &= \widetilde{h} + (\gamma - 1)s \quad \mbox{for any $\gamma \in \Gamma$}.
	\end{align*}
	Since the isomorphism class $[\widetilde{\cH}]$ is independent of the choice of $\widetilde{\cH}$ and $s$, 
	we obtain a map
	\[
		\Ext(\Delta, D) \to \Ext(\cR, \Hom_{\Rmod}(\Delta,D)) \cln [\widetilde{D}] \mapsto [\widetilde{\cH}],
	\]
	which gives the inverse of (\ref{isomBetweenTwoExts}).
	
	For $F$-analytic $(\vphi_q, \Gamma)$-modules $D_1$ and $D_2$ over $\cR$,
	$\Ext_{\an}(D_1,D_2)$ denotes the $L$-subspace of $\Ext(D_1,D_2)$ consisting of $F$-analytic extensions.
	
	\begin{lem}\label{ExtLemma}
		Let $\Delta$ and $D$ be $F$-analytic $(\vphi_q, \Gamma)$-modules over $\cR$.
		Then the image of $\Ext_{\an}(\cR, \Hom_{\Rmod}(\Delta,D))$ by the isomorphism (\ref{isomBetweenTwoExts}) is $\Ext_{\an}(\Delta, D)$.
	\end{lem}
	
	\begin{proof}
		Take any extension $0 \to \Hom_{\Rmod}(\Delta, D) \to \widetilde{\cH} \to \cR \to 0$, and suppose that, by the isomorphism (\ref{isomBetweenTwoExts}), the class $[\widetilde{\cH}]$ maps to $[\widetilde{D}] \in \Ext(\Delta, D)$.
		First, we assume that $\widetilde{\cH}$ is $F$-analytic.
		Then we must show that $\Delta$ as a $\Gamma$-submodule of $\widetilde{D}$ is $F$-analytic.
		We choose a lift $\widetilde{h}$ of $1\in\cR$ in $\widetilde{\cH}$ as above.
		Take any $\gamma\in\Gamma$ sufficiently close to 1.
		By Lemma \ref{calculationFormulaOfLog}, the element $(\log\gamma) \widetilde{h}$ is indeed in $\Hom_{\Rmod}(\Delta,D)$.
		By using Lemma \ref{calculationFormulaOfLog} again, for any $x\in \Delta$, we have
		\begin{align*}
			(\log \gamma)|_{\widetilde{D}} (x) &= (\log \gamma)|_{\Delta}(x) + [(\log\gamma)(\widetilde{h})] (x) \ \ \mbox{and} \\
			\left. \frac{\log \gamma}{\log\chi_\LT (\gamma)} \right|_{\widetilde{D}} (x)
				&=  \nabla|_{\Delta}(x) + [\nabla\widetilde{h}](x).
		\end{align*}
		Therefore the operator $\log\gamma/(\log\chi_\LT(\gamma))$ on $\widetilde{D}$ is independent of the choice of $\gamma$ and $\widetilde{D}$ is $F$-analytic.
		
		Conversely, we will show that the $F$-analyticity of $\widetilde{D}$ yields that of $\widetilde{\cH}$.
		Take any section $s\cln \Delta \to \widetilde{D}$ of the projection $\widetilde{D} \to \Delta$ as an $\cR$-module.
		Since both $\Delta$ and $D$ are $F$-analytic, the $(\vphi_q, \Gamma)$-module $\Hom_{\Rmod}(\Delta, D)$ is $F$-analytic.
		Thus $\nabla$ on $\Hom_{\Rmod}(\Delta, D)$ is well-defined.
		By Lemma \ref{calculationFormulaOfLog}, we have
		\[
			(\log \gamma)\widetilde{h} = (\log \gamma)s \ \ \mbox{and }\ 
			\frac{\log\gamma}{\log\chi_\LT(\gamma)} \widetilde{h} = \nabla s.
		\]
		Therefore $[\log\gamma/(\log\chi_\LT(\gamma))]\widetilde{h}$ is independent of the choice of $\gamma$ and $\widetilde{\cH}$ is $F$-analytic.
	\end{proof}
	
	By Lemma \ref{ExtLemma} and \cite[Cor.\ 4.4]{FX2013}, we have the following:
	\begin{cor}\label{ExtAndExtan}
		Let $\Delta$ and $D$ be $F$-analytic $(\vphi_q, \Gamma)$-modules over $\cR$.
		Then $\Ext_{\an}(\Delta, D)$ is of codimension
		\[
			[F:\Q_p] \dim_L \Hom_{\pGmod}(\Delta, D)
		\]
		in $\Ext(\Delta,D)$.
		In particular, if $\Hom_{\pGmod}(\Delta, D) = 0$, then we have
		\[
			\Ext_{\an}(\Delta, D) = \Ext(\Delta,D).
		\]
	\end{cor}
	
	Now we state and prove the main theorem.
	For a $(\vphi_q, \Gamma)$-module $D$ over $\cR$, a finite extension $L'\subset \Fb$ of $L$ and a character $\delta\cln F^\times \to (L')^\times$,
	we denote by $D(\delta)$ the $(\vphi_q, \Gamma)$-module over $\cR_{L'}$ whose underlying $\cR_{L'}$-module is $L'\otimes_{L}D$ and
	whose $(\vphi_q, \Gamma)$-action is defined by
	\begin{align*}
		\vphi_q|_{D(\delta)} (x) &= \delta(\vpi) (\id_{L'}\otimes \vphi_q|_D) (x) \mbox{\qquad and} \\
		\gamma |_{D(\delta)} (x) &= \delta(\chi_\LT(\gamma)) (\id_{L'}\otimes \gamma_q|_D) (x)
	\end{align*}
	for any $x\in D(\delta)$ and any $\gamma\in \Gamma$.
	\begin{thm}\label{trueMainThm}
		Let $D$ be a $(\vphi_q, \Gamma)$-module over $\cR$.
		Suppose that $D$ has a filtration
		\[
			0 = D_0 \subset D_1 \subset \cdots \subset D_r = D
		\]
		of $(\vphi_q, \Gamma)$-modules over $\cR$.
		We put $\Delta_i = D_i/D_{i-1}$.
		We assume that the following conditions hold.
		
		\begin{enumerate}
			\renewcommand{\theenumi}{{\rm (\alph{enumi})}}
			\item For any $1 \leq i \leq r,$ we have $\End_{\pGmod} (\Delta_i) = L$.
			\item For any $1 \leq i<j \leq r$, we have $\Hom_{\pGmod} (\Delta_j,\Delta_i) = 0$.
			\item For any $1 < i \leq r,$ the short exact sequence 
			$ 0 \to D_{i-1} \to D_i \to \Delta_i \to 0$ does not split.
		\end{enumerate}
		Then there exist a finite extension $L'$ of $L$ and a character $\delta\cln F^\times \to (L')^\times$ such that,
		for any $0 \leq i \leq r,$ the $(\vphi_q, \Gamma)$-module $D_i(\delta)$ is $F$-analytic.
	\end{thm}
	
	\begin{proof}
		We prove it by induction on $r$.
		First, we prove the case $r=1$.
		Let $D$ be a $(\vphi_q, \Gamma)$-module over $\cR$ such that $\End_{\pGmod} (D) =L$.
		We choose a $\Z_p$-basis $\beta_1,\ldots,\beta_n$ of $\Lie\Gamma$.
		For any $1\leq i \leq n$, we have $\nabla_{\beta_1}-\nabla_{\beta_i} \in \End_{\pGmod}(D).$
		Thus there exist $c_1=0,c_2,\ldots,c_n\in L$ such that $\nabla_{\beta_1} - \nabla_{\beta_i} = c_i \cdot \id_D$ for any $1\leq i \leq n$.
		
		Now we choose a $\Z_p$-basis $\gamma_1, \ldots, \gamma_n$ of the free part of $\Gamma$ such that $\gamma_i^{m_i} = \exp(p^{m_i} \beta_i)$ for some integers $m_i\geq 0$.
		We fix a finite extension $L'$ of $L$ and a character $\delta\cln F^\times \to (L')^\times$.
		By Lemma \ref{calculationFormulaOfLog} , for any $x \in D(\delta)$, we can compute
		\[
			\nabla_{\beta_i}|_{D(\delta)} (x) = \left( \id_{L'}\otimes \nabla_{\beta_i}|_D\right) (x) + \frac{\log\delta(\chi_\LT(\gamma_i))}{\log\chi_\LT(\gamma_i)} x. 
		\]
		Thus, if we find a character $\delta$ such that the equality
		\[
			c_i + \frac{\log\delta(\chi_\LT(\gamma_1))}{\log \chi_\LT(\gamma_1)} - \frac{\log\delta(\chi_\LT(\gamma_i))}{\log \chi_\LT(\gamma_i)} = 0
		\]
		holds for each $1\leq i \leq n$, then $D(\delta)$ is $F$-analytic.
		Actually, we can do it: on the torsion part of $F^\times$ and on $\chi_\LT(\gamma_1)$, put $\delta = 1$. On $\delta(\chi_\LT(\gamma_2)), \ldots, \delta(\chi_\LT(\gamma_n))$, put
		\[
			\delta(\chi_\LT(\gamma_i^{e_i})) = \exp\left( c_i \log\chi_\LT(\gamma_i^{e_i}) \right),
		\]
		where $e_i$ is a sufficiently large integer such that the right-hand side is well-defined.
		Finally, define $\delta(\chi_\LT(\gamma_i))$ as an $e_i$-th root of $\delta(\chi_\LT(\gamma_i))$
		(hence we must extend $L$ in general).
		
		Next we suppose that the statement for $r-1$ holds.
		By extending scalar and twisting by a character, we may assume that, for any $1\leq i \leq r-1$, $D_i$ is $F$-analytic.
		Then we must show that $D_r$ is $F$-analytic.
		By the assumption (a), (c) and Lemma \ref{nonTrivExt}, $\Delta_r$ is $F$-analytic.
		Now we apply the functor $\Hom_{\pGmod}(\Delta_r,\bullet)$ to the short exact sequences
		\[\begin{alignedat}{5}
			 0 & \to \Delta_1 &&\to D_2 &&\to \Delta_2 &&\to 0,  \\
			 0 &\to D_2 &&\to D_3 &&\to \Delta_2 &&\to 0,  \\
			  & && \hspace{7mm} \rotatebox{90}{$\cdots$} && &&  \\
			 0 &\to D_{r-1} &&\to D_r &&\to \Delta_r &&\to 0.  \\
		\end{alignedat}\]
		Then, by the assumption (b), we inductively obtain $\Hom_{\pGmod}(\Delta_r, D_i) = 0$ for any $1 \leq i \leq r-1$.
		By Corollary \ref{ExtAndExtan}, we have $\Ext_{\an}(\Delta_r, D_{r-1}) = \Ext(\Delta_r, D_{r-1}).$
		Therefore $D_r$ is $F$-analytic and we have the conclusion.
	\end{proof}
	
	This theorem, together with Theorem \ref{equivBetTwoAnalyticity}, implies the following:
	
	\begin{cor}[{Theorem \ref{MainThm}}]
		Let $V$ be an overconvergent $L$-representation of $\sG_F$ such that $\cR\otimes_{\cE^\dagger} \bfD^\dagger(V)$ satisfies the assumptions of Theorem \ref{trueMainThm}.
		Then there exist a finite extension $L'$ of $L$ and a character $\delta\cln \sG_F \to (L')^\times$ such that
		$V\otimes_L L'(\delta)$ is $F$-analytic.
	\end{cor}
	
	\providecommand{\bysame}{\leavevmode\hbox to3em{\hrulefill}\thinspace}
\providecommand{\MR}{\relax\ifhmode\unskip\space\fi MR }
\providecommand{\MRhref}[2]{%
  \href{http://www.ams.org/mathscinet-getitem?mr=#1}{#2}
}
\providecommand{\href}[2]{#2}

\end{document}